\documentclass[11pt]{article}
\usepackage[top=1in,bottom=1in,left=1in,right=1in]{geometry}
\usepackage{graphicx}
\usepackage{amsmath,amssymb,amsthm}

\ifx\theorem\undefined
\newtheorem{theorem}{Theorem}
\fi
\ifx\example\undefined
\newtheorem{example}{Example}
\fi
\ifx\property\undefined
\newtheorem{property}{Property}
\fi
\ifx\lemma\undefined
\newtheorem{lemma}[theorem]{Lemma}
\fi
\ifx\proposition\undefined
\newtheorem{proposition}[theorem]{Proposition}
\fi
\ifx\remark\undefined
\newtheorem{remark}[theorem]{Remark}
\fi
\ifx\corollary\undefined
\newtheorem{corollary}[theorem]{Corollary}
\fi
\ifx\definition\undefined
\newtheorem{definition}[theorem]{Definition}
\fi

\begin{document}
\bibliographystyle{plain}

\title{Learning Patterns for Detection with Multiscale Scan Statistics}
\author{
James Sharpnack \\
{\tt  jsharpna@ucdavis.edu} \\
Department of Statistics, UC Davis\\
}

\date{}%\today}
\maketitle

\begin{abstract}
This paper addresses detecting anomalous patterns in images, time-series, and tensor data when the location and scale of the pattern and the pattern itself is unknown a priori.
The multiscale scan statistic convolves the proposed pattern with the image at various scales and returns the maximum of the resulting tensor.
Scale corrected multiscale scan statistics apply different standardizations at each scale, and the limiting distribution under the null hypothesis---that the data is only noise---is known for smooth patterns. 
We consider the problem of simultaneously learning and detecting the anomalous pattern from a dictionary of smooth patterns and a database of many tensors.
To this end, we show that the multiscale scan statistic is a subexponential random variable, and prove a chaining lemma for standardized suprema, which may be of independent interest.
Then by averaging the statistics over the database of tensors we can learn the pattern and obtain Bernstein-type error bounds.
We will also provide a construction of an $\epsilon$-net of the location and scale parameters, providing a computationally tractable approximation with similar error bounds.
\end{abstract}

\section{Introduction}

Detection is the statistical task of determining if there is some structured signal within noisy data.
If classification answers the question, ``what am I seeing?'', detection answers the question, ``do I see anything at all?''.
In a sensor network (see \cite{culler2004overview}), one is often interested in the dual problems of noticing an anomaly (detection) and then determining its location and extent (classification).
Sensor networks are deployed in natural environments for contaminant detection, \cite{yang2009real, white2008contaminant}, real-time surveillance (\cite{caron2002method}), radiation monitoring (\cite{brennan2004radiation}), and fire detection (\cite{pozo1997fire}).  In medical imaging, the critical task is often to test the existence of an anomaly (\cite{moon2002brain, james2001breast}).
Quick detection of outbreaks of pathogens, \cite{heffernan2004ssp,rotz2004advances}, can lead to early intervention.
Yet, given that this is such a fundamental task in many applications, the development of detection methodology lags behind the sophisticated tools for classification.

In images, time series, and tensors, it is natural to assume that there is a structured signal, such as blob-like objects, but we do not know its location or scale.
In Figure \ref{fig:scan}, we can see a chemical plume from a multispectral image where each pixel value is lighter if a certain spectral signature is present (see \cite{manolakis2002detection}).
We did not know a-priori where this chemical would appear, nor how large it would be.
A natural approach to detecting signals of this type would be to center a regular shape, or pattern, such as a square or an ellipse around every pixel and detect anomalously large quantities of the spectral signature.
In this work, we will address the dual problem of detecting anomolous patterns in tensors, and learning regularly occuring patterns in a tensor dataset from a dictionary. 
Let's begin by outlining multiscale scan statistics and the recent advances in adaptive detection.

\begin{figure}[ht]
    \centering
    \includegraphics[width=.45\textwidth]{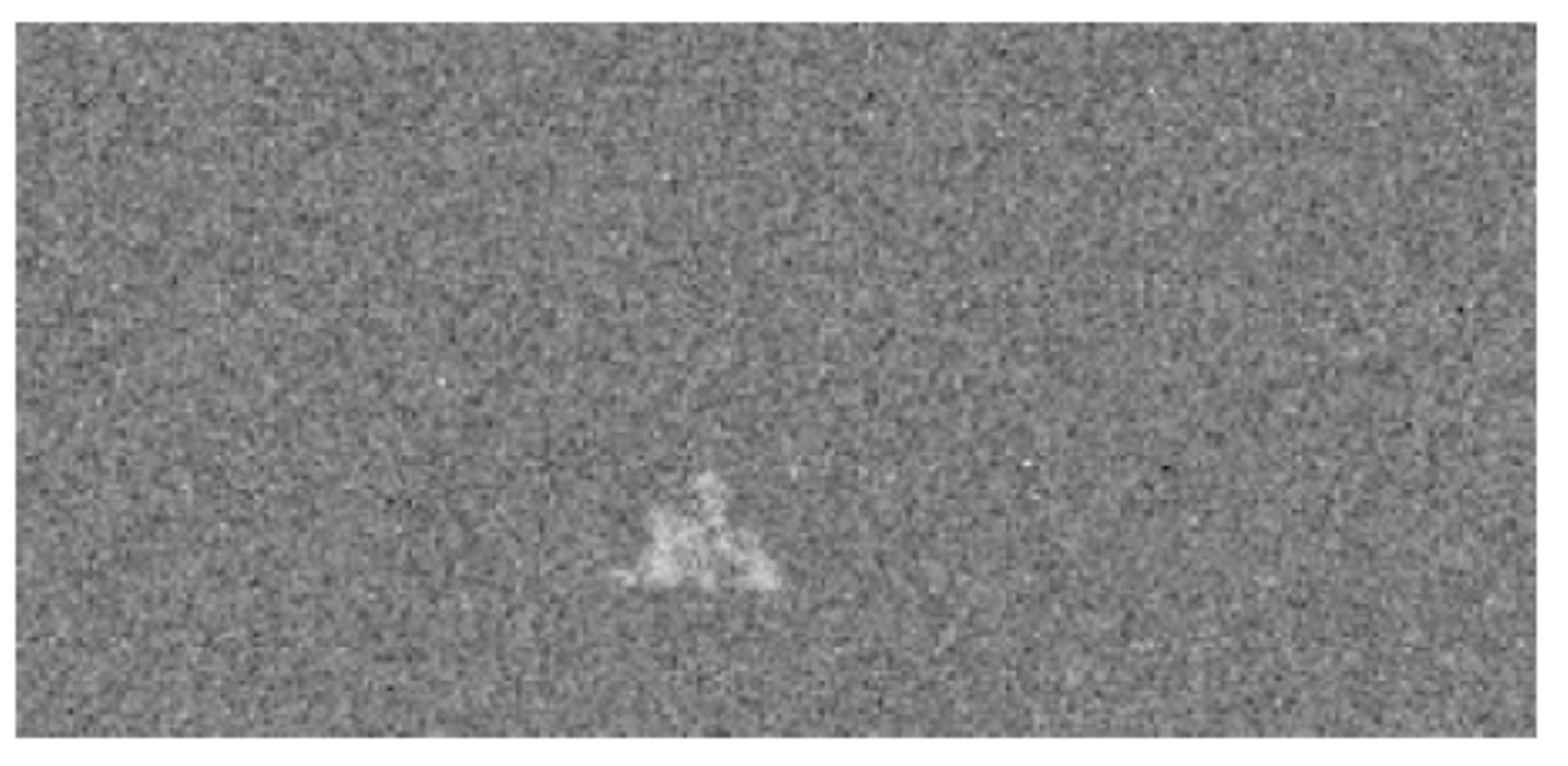}\hspace{.05\textwidth}
    \includegraphics[width=.45\textwidth]{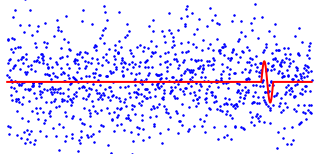}
    \caption{An image with an anomalous region of contaminant (left) and a simulated time series with an embedded sinusoidal signal with values on the $y$-axis (right).}
    \label{fig:scan}
\end{figure}

The aptly named scan statistic is a test statistic that is based on scanning the image or time series ($X$) for a pattern ($f$) that may be centered anywhere in the domain.
For each location ($t$), one can form a likelihood ratio, and test if this statistic exceeds some predetermined threshold (thus, it is a generalized likelihood ratio test).
Scan statistics are widely used in spatial detection applications; see \cite{glaz2001scan} for a thorough introduction to the topic.
They are commonly used to detect patterns in point clouds, \cite{Naus65, neill2012fast}, a closely related problem to our own.
Because we often do not know the size as well as the location of the anomaly, we will scan all locations ($t$) and scales ($h$) simultaneously---this is called the {\em multiscale scan statistic}. 
Hence, the multiscale scan will translate various scaled versions of a pattern, such as rectangles with many options of side lengths or circles of varying radii.
In Neyman-Pearson testing, we attempt to control the probability of false rejection under a null hypothesis, when our data only consists of noise.
To this end, we either must approximate the distribution of the scan statistic asymptotically, or set the detection threshold by simulation or resampling.  
\cite{siegmund1995testing} provided a weak limit for the scan consisting of all intervals in 1 dimension (other approximations can be found in \cite{naus2004multiple} and \cite{pozdnyakov2005martingale}).
In 2 dimensions, \cite{glaz2004multiple, haiman2006estimation, wang2014variable, kabluchko2011extremes}, provided approximations of the null distribution for the multiscale scan statistic.
\cite{MGD, cluster} analyzed scan statistics for blob-like patterns and determined thresholds for detectability in this context. 

It was observed that if one naively tests all scales in the multiscale scan at the same threshold, then the rejection events will be dominated by the finest scale (\cite{chan2013detection}).
It was shown in \cite{dumbgen2001multiscale} that by separately standardizing the 1-dimensional scan at each scale one can detect with a signal-to-noise ratio that adapts to the scale of the anomalous signal.
This scale correction for multiscale scan statistics with rectangular shapes was further studied in \cite{MR2604703, sharpnack2016exact}.
\cite{proksch2016multiscale} showed that similar results can be achieved for a large class of patterns that satisfy an average H\"older continuity assumption.

% In \cite{sharpnack2016exact}, the authors demonstrated that an $\epsilon$-net for the rectangular scan could be formed by a hierarchy of downsampling and convolutional operations, and were able to used fast GPU implementations.
% Convolutional neural networks, \cite{lecun1995convolutional}, consist of hierarchies of convolutional layers composed with non-linear operations.
% Despite the fact that the resulting optimization programs are non-convex, methods like stochastic gradient descent can allow us to learn features for classification.
% Deep learning toolkits provide a promising avenue for developing fast scan statistics, but the analogy to the convolutional neural network also suggests that we might learn representations in the detection setting. 
% %More recently, deep learning has emerged as an essential tool in the machine learning toolbox, with customized implementations on GPUs (\cite{krizhevsky2012imagenet}).
% Because the non-linear max operation is integral to scan statistics and detection in general, these developments beg the question: is it possible to learn the detectable patterns from a large database of images?

In traditional detection applications, a tensor is scanned for a known function or for specific blob-like patterns.
In a database of many tensors with a repeated anomalous pattern in a database then it may be possible to simultaneously test for all of the patterns in a class $\mathcal F$.
The catch is that each time-series or image---more generally, tensor---will have the anomalous pattern in a different location and at a different scale.
For example, we have many time series, all of which have some embedded smooth signal (such as the sinusoid in Figure \ref{fig:scan}) where the sinusoid begins at various time points, and has different periodicities.
Without knowing a priori that we are looking for a sinusoid this problem seems intractable, but given that we know enough about the signal (such as it comes from a finite dictionary of functions), then with enough data we may both learn and detect the pattern within very noisy data.

Previous works have also studied the detection of any pattern from a large class of patterns in a single noisy tensor and other more complicated detection scenarios.
The main difference between these works and our own is in the order of operations, namely our methodology proposes a pattern $f$, then tests if this patterns occurs in the tensors $X^i$ for all $i$ in a multiscale fashion, and repeats this for all patterns in a class; most methods will consider a tensor $X^i$ (or a subset of tensors) then tests for any pattern within a class.
\cite{cluster} studied the detection of Lipschitz deformations of the Euclidean ball in a single image.
One can also consider a time series of images to form a single order 3 tensor; this is studied in \cite{kifer2004detecting}, where they detect temporal changepoints of any patterns within a VC class.
Detecting intervals in multiple time-series has also been studied in \cite{chan2015optimal}, but this setting differs from our own because it assumes that only a sparse subset of the time series will contain an anomalous interval.

\subsection{Contributions}

We will introduce the multiscale scan statistic, with scale correction, in Section \ref{sec:method}, and describe an $\epsilon$-net construction using repeated dilation and convolution operations.
We begin our theoretical analysis, in Section \ref{sec:chaining}, by proving a chaining result which shows that the standardized supremum of certain subGaussian random fields is a subexponential random variable.
This is then used, in Section \ref{sec:main} to provide a finite sample bound on the scale corrected multiscale scan statistic (until now it was only known that it was A.S. bounded).
We conclude by demonstrating that our $\epsilon$-net construction is indeed correct and a type 2 error control in Section \ref{sec:eps}.

\section{Method and Model}
\label{sec:method}

\subsection{Continuous scan statistic}

Let's begin with the basic scan statistic over an image.
For an image, which can be represented by a matrix, $\{Y_{k,l} :k,l=-L,\ldots,L\}$ we can convolve a pattern $\{P_{k,l}:k,l=-H,\ldots,H\}$ with the image,
$$
(P \star Y)_{k,l} = \sum_{k',l' = -H}^{H} Y_{k - k',l - l'} P_{k', l'}, \quad k,l = -L + H, \ldots, L-H.
$$
Then the simple scan statistic is $\max_{k,l} (P \star Y)_{k,l}$ (one can take the absolute value by also considering $P \gets -P$).
In the case that $P_{k,l} = 1/(2H)$ then this scans a square activation pattern over the image.
A common assumption in this problem, is that under the null hypothesis, $Y_{k,l}$ are independent standard subGaussian random variables.
For an arbitrary, pattern matrix $P$ we would like to scale both dimensions, so that $H \gets H_j'$ in dimension $j$ for some $1 \le H_j' \le L$ (e.g.~ stretching the square to form a rectangle).
For general patterns, this can be difficult, and one approach is to scale the dimensions for a continuous function, $f$, and then rasterize it.
This analysis can become very cumbersome and not terribly enlightening, so we will approximate the pixels with a continuous domain and the image with a random field.

In order to implement our scan, we begin by proposing a given pattern, which is a function $f \in \mathcal F$ over the $d$-dimensional domain.
For images, $d=2$, and for time-series, $d=1$, but we will only assume that the $d$ is fixed in our asymptotics.
We assume that for every $f \in \mathcal F$, $\| f \|_{L_2} = 1$ and is supported over $\Omega := [-1,1]^d$.
We will further assume that $f$ has continuous gradient ($\mathcal F \subset C^1(\Omega)$).
For a given field, $X^i$, over $\Omega_L := [-L,L]^d$, we propose a scale parameter, 
$$
h \in \mathcal H := \times_j [1, L),
$$
such that $h_j$ is the scale parameter for dimension $j$.
(Throughout $j=1,\ldots d$ and $j$ will always indicate dimension.)
Given an $h \in \mathcal H$, we select 
$$
t \in \mathcal T_h:= \times_j [- (L - h_j), L-h_j]
$$ 
respectively, and test if the pattern $f$ centered at $t$ and scaled by $h$ is hidden within image $X^i$.
This is accomplished by convolving the field, $X^i$, with the scaled function $f_h: = h_\bullet^{-1/2} f(. / h)$,
%\begin{equation}
$$
(f_h \star {\rm d} X^i)(t) = \int_{\Omega_L} f_h(\tau) {\rm d} X^i(t - \tau) = \int_{\Omega} \frac{1}{\sqrt{h_\bullet}}f(\tau) {\rm d} X^i(t - h \tau),
$$
%\end{equation}
where $h_\bullet = \prod_j h_j$ and vector operations are such that $t/h$ are performed elementwise.
(Throughout, $f_h(t) = 0$ if $t$ is outside of the domain of $f_h$.)
The multiscale scan statistic takes the form,
\begin{equation}
\label{eq:MSS}
s(X^i; f) := \max_{h \in \mathcal H} v_h \left( \max_{t \in \mathcal T_h} (f_h \star {\rm d} X^i)(t) - v_h \right).
\end{equation}
where in this work we will take $v_h = \sqrt{2 \sum_j \log (L/h_j)}$.
Hence, when the scale is coarse ($h_j$ is large) then $v_h$ is smaller.
At the finest scale, $v_h$ is large, the maximum at this scale concentrates about $v_h$ with a rate parameter $v_h$ (meaning that this concentrates more tightly than the pixel noise variance of $1$).

Throughout, we will denote the set of all valid scale and location parameters as
$$
\mathcal D := \{(t,h) : h \in \mathcal H, t \in \mathcal T_h \}.
$$
We will assume that the fields, $X^i$, are independent and have additive white noise terms, $W^i$.
Our noise model is such that $\int f {\rm d} W^i$ is a zero mean subGaussian random field indexed by $f$ with variance $\| f\|_{L_2}^2$ (see Section \ref{sec:chaining} for a definition).
We will test if the field $X^i$ is just noise (null hypothesis) or the noise is added to the function, $f$, which is translated by $t^i$ and scaled by $h^i$,
\begin{align*}
    H_0&: {\rm d} X^i(\tau) = {\rm d} W^i(\tau), i=1,\ldots,n \\
    H_1&: {\rm d} X^i(\tau) = \mu f_{h^i}(t^i - \tau) {\rm d} \tau + {\rm d} W^i(\tau) \\
    &\quad {\rm ~for~some~} f \in \mathcal F, {\rm~and~} (t^i, h^i) \in \mathcal D, i =1,\ldots,n.
\end{align*}
Critically, the location and scale parameters $t^i, h^i$, differ for each image.  
Notice that $\|f_h\|_{L_2} = 1$, so under the null hypothesis, the convolution $(f_h \star {\rm d} X^i)(t)$ has mean $0$ and variance $1$.
Because each statistic, $s(X^i; f)$, is independent and standardized by $v_h$, then we can average these in order to increase the power of our final test statistic.
To this end, let's define the pattern adapted multiscale scan statistic (PAMSS), 
\begin{equation}
S_n(X; \mathcal F) := \max_{f \in \mathcal F} \frac{1}{\sqrt n} \sum_{i=1}^n s_n(X^i; f).
\end{equation}
We will show that $s_n(X^i;f)$ is subexponential, and so we can obtain probabilistic bounds on $S_n$ with the subexponential Bernstein-type inequality.

\subsection{Multiscale $\epsilon$-net construction}

In practice, data is discrete, and the scan statistic must be computed over a finite set of scales and locations.
Suppose that for sample, $X^i$, we have a draw from the null hypothesis, so that $\mathbb E {\rm d} X^i(\tau) = \mu f_{h^i} (t^i - \tau)$.
Instead of scanning over all $(t,h) \in \mathcal D$, we use a finite subset, $\mathcal D_{\rm net} \subset \mathcal D$.
Let $t',h'$ be values in $\mathcal D_{\rm net}$ that are close to $t^i, h^i$ in some sense.
Then the expectation of the scan at this approximating location and scale is
$$
\mathbb E (f_{h'} \star {\rm d}X^i)(t') = \mu \int f_{h'}(\tau) f_{h^i} (\tau + t^i - t') {\rm d} \tau.
$$
If we define the shift operator $(S_t f)(\tau) := f(\tau + t)$ then we see that this expectation is $\mu \langle S_{t'} f_{h'} ,S_{t^i} f_{h^i} \rangle_{L_2}$.
Incidentally, this metric appears when we consider the variation of our scan statistic under the null hypothesis,
$$
\nu_f((t,h),(t',h')) := \left\| S_{t} f_h - S_{t'} f_{h'} \right\|_{L_2} = \left(\mathbb V \left( (f_h \star {\rm d} W)(t) - (f_{h'} \star {\rm d} W)(t') \right)  \right)^{\frac 12},
$$
where ${\rm d}W$ satisfies our noise assumptions (such as the standard multivariate Brownian motion).
This fact will be useful for when we provide type 1 error control for our scan statistic.
With this metric we will say that a finite subset $\mathcal D_{\rm net} \subset \mathcal D$ is an {\em $\epsilon$-net} if for any $f\in \mathcal F$ and any point $(t,h) \in \mathcal D$ there exist points $(t',h') \in \mathcal D_{\rm net}$ such that $\nu_f((t,h),(t',h')) \le \epsilon$.

The sensitivity of $\nu_f$ to small changes in $t,h$ will depend on the smoothness of $f$.
If $f$ has large gradients then a small shift, $S_t$, can misalign the function with the unshifted version.
To this end, we will consider two different notions of smoothness for the functions in our dictionary, $\mathcal F$.
Define the isotropic total variation (recall that the functions have continuous gradients),
$$
\| f \|_{\rm TV} := \int_\Omega \| \nabla f(u) \|_2 {\rm d} u.
%\sup \left\{ \int_\Omega f {\rm div} g ; g \in C_c^1(\Omega, \mathbb R^d), \| g\|_\infty \le 1 \right\}
$$
Then we may assume that all of the function are of bounded variation,
\begin{equation}
    \label{eq:TVC}
    \exists \gamma_1 > 0 {\rm ~s.t.~} \forall f \in \mathcal F, \quad \| f \|_{\rm TV} \le \gamma_1. \tag{TVC}
\end{equation}
The bounded variation assumption is consistent with the assumption in \cite{dumbgen2001multiscale}, although we generalize this work to the multidimensional case.
Another notion of smoothness is the average H\"older condition of \cite{proksch2016multiscale}.
Define the H\"older functional,
$$
A_{t,s}(f) := \int_{\Omega_L} \left|f(t-z) - f(s-z) \right|^2 {\rm d} z.
$$
Then an alternative to the bounded variation assumption is that
\begin{equation}
    \label{eq:AHC}
    \exists 0 < \gamma_2 \le 1 {\rm ~s.t.~} \forall f \in \mathcal F,\quad A_{t,s}(f) \le c_A \| t-s\|_2^{2 \gamma_2} \tag{AHC}
\end{equation}
where $c_A$ is some constant.  Throughout the paper we will refer to both conditions, and denote $\gamma$ as either $\gamma_1$ or $\gamma_2$ depending on context.

Given that our functions $\mathcal F$ satisfy some smoothness assumptions, we can specify an $\epsilon$-net construction.
For a given $\beta > 1$ define $\ell_{\max} = \lfloor \log_\beta L \rfloor$.
We will begin with a subset of scales,
$$
\mathcal H_\beta = \{ \beta^\ell : \ell = \{0,\ldots, \ell_{\max}\}^d \}.
$$
where $\beta > 1$ is some parameter.
At a given scale, we will consider a grid of evenly spaced locations, where the distances between grid points increases with the scale, $h$,
$$
\mathcal T_{\alpha,h} := \left( \times_j (\alpha h_j \cdot \mathbb Z) \right) \cap \mathcal T_h.
%= \times_j \{ -(\ell_{\max} - \alpha) h_j, -(\ell_{\max} - 2 \alpha) h_j, \ldots, -h_j, 0, h_j,\ldots, (\ell_{\max} - 2\alpha) h_j, (\ell_{\max} - \alpha) h_j \}.
$$
The spacing, $\alpha > 0$, is a tuning parameter as well in this construction.
With $\beta, \alpha$ specified, we can consider the $\epsilon$-net to be
$$
\mathcal D_{\beta,\alpha} = \{ (t,h) : h \in \mathcal H_\beta, t \in \mathcal T_{\alpha,h} \}.
$$
(Notationally, we exchange $\mathcal D_{\rm net} \gets \mathcal D_{\beta,\alpha}$.)
This is similar to the construction in \cite{sharpnack2016exact}, which has a fast implementation on GPUs using a hierarchy of convolution and downsampling layers.
The main idea is that instead of expanding the function $f \rightarrow f_h$ can instead downsample the tensor (but the details on a finite image are somewhat arduous).
We expect that a similar implementation is possible in this context, when the functions are adaptively rasterized, but such developments are outside of the scope of this paper.

Given an $\epsilon$-net, we can compute the approximate scan by restricting the evaluations to the finite set of location and scales,
$$
e_{\beta,\alpha}(X^i; f) := \max_{(t,h) \in \mathcal D_{\beta,\alpha}} v_h \left( (f_h \star X^i) (t) - v_h \right).
$$
Then we similarly can define the $\epsilon$-net pattern adapted multiscale scan statistic ($\epsilon$-PAMSS),
$$
E_{n}(X; \mathcal F) = \max_{f \in \mathcal F} \frac{1}{\sqrt n} \sum_{i=1}^n e_{\beta,\alpha}(X^i; f).
$$
It is clear that $E_n \le S_n$ since we are maximizing over a strict subset of the continuous scan.
Hence, any type 1 error bound on $S_n$ is also conferred to $E_n$.

\section{A chaining bound for standardized suprema}
\label{sec:chaining}

The scale correction in \eqref{eq:MSS} is based on a precise characterization of the rate and location of a supremum of the random field resulting from the convolution.
This standardization complicates the analysis significantly, and until now, it was only known that $s(X^i;f)$ was almost surely bounded.
We initially select the location and scale parameters $(t,h) \in \mathcal D$ and take the supremum over this selection in forming the scan of $f$ over $X^i$.
Because under $H_0$, the random variables $s(X^i; f)$ form iid copies of the same random variable, we will seek to exponential concentration inequalities bounding $s$.
With this in hand, we can hope to obtain PAC-style bounds on their average and the resulting selection of pattern from a finite function class $\mathcal F$.
We will restrict this work to $|\mathcal F| < \infty$, but a continuous class of location and scale parameters, $\mathcal D$.
While this bar seems to be set pretty low compared to the rich developments in classification, such as Vapnik-Chernovenkis theory, as we will see, controlling our statistic under finite function classes is a challenging first step to a more complete understanding of learning patterns in detection.
Let us begin with a formal definition of a subGaussian random field, and recall that this is what we assume for the field, $\{\int f_h {\rm d} W^i(t-.) : t,h \in \mathcal D \}$, which satisfied when $W^i$ is the $d$-dimensional Weiner process. 

\begin{definition}
We say that a random field, $\{Z(\iota)\}_{\iota \in \mathcal I}$, is a (zero mean) standard subGaussian process if there exists a constant $u_0 > 0$ such that
\begin{align}
\label{eq:subGaussian1}
    &\mathbb P \left\{ |Z(\iota_0) - Z(\iota_1) | \ge u \right\} \le 2 \exp \left( - \frac{u^2}{2 d_Z(\iota_0,\iota_1)} \right),\\
\label{eq:subGaussian2}
    &\mathbb P \left\{ Z(\iota_0) \ge u \right\} \le \exp \left( - \frac{u^2}{2} \right),
\end{align}
for any $\iota_0, \iota_1 \in \mathcal I$, $u > u_0$, and $d_Z(\iota_0,\iota_1) = \sqrt{\mathbb E (Z(\iota_0) - Z(\iota_1))^2}$,    
is the canonical distance.
\end{definition}
Our noise model assumption can be formally stated as $\{(f_h \star {\rm d} W^i)(t): (t,h) \in \mathcal D \}$ is a subGaussian random field with canonical distance $\nu_f$.
%This assumption is satisfied when $W^i$ the standard multivariate Wiener process.

The generic chaining is a tool for bounding the suprema of random fields with subGaussian tails, \cite{talagrand2006generic} (it also has generalizations to other types of concentration).
For general subGaussian random fields, the expectation of the supremum is bounded by a quantity, $\mathbb E \sup Z \le T_Z$, and one can also show that $(\sup Z - T_Z) / T_Z$ is a standard subGaussian random variable (using the standard chain construction). 
Hence, as the expectation bound, $T_Z$, increases the scale can also grow with $T_Z$ (meaning that the supremum becomes more dispersed).
This is in contrast to what we know from extreme value theory, where the maximum of independent random variables tends to concentrate more tightly (with a lower asymptotic variance), not less.
For example, let $\{z_i\}_{i=1}^N$ be iid standard normal random variables.
Then we know that $b_N (\max_i z_i - b_N)$ approaches a Gumbel distribution where $b_N = \sqrt{2 \log n} + o(1)$ is a specific sequence (see \cite{de2007extreme}).
Notice that the scale of the max decreases like $1 / b_N$, which is in contrast to what we obtain from the standard construction in the generic chaining, which has increasing scale.
Multiscale scan statistics also have a Gumbel limiting distribution, \cite{sharpnack2016exact}, and so we know that the scale of the statistic should be decreasing, and not increasing.
The following theorem significantly modifies the construction in the generic chaining in order to provide an exponential inequality under VC-style conditions on the entropy of the random field, $Z$.

\begin{theorem}
\label{thm:chaining}
Let $Z(\iota)$ be a standard subGaussian process over an index set $\mathcal I$.
Suppose that the metric space $(\mathcal I, d_Z)$ has the following bound on the $\epsilon$-covering number ($\rho > 0$),
\begin{equation}
    \mathcal N(\mathcal I,d_Z, \epsilon) \le \Gamma \epsilon^{-\rho}.
\end{equation}
Then there exists an $\Gamma_0 > 0$ such that for any $\Gamma \ge \Gamma_0$, the following supremum is bounded in probability,
\begin{equation}
    \mathbb P \left\{ \sqrt{c_0 \log \Gamma} \left( \sup_{\iota \in \mathcal I} Z(\iota) - \sqrt{2 \log \Gamma} \right) - a_0 \log \log \Gamma > u \right\} \le e^{-u},
\end{equation}
for $u > u_0$ where $u_0, c_0, a_0$ are constant depending on $\rho$ (but not on $\Gamma$).
In words, the supremum of such a subGaussian process is subexponential with location and rate parameter, $(2 \log \Gamma)^{1/2}$. 
\end{theorem}

\begin{remark}
For a VC class of sets with metric $L_2(Q)$ for some measure $Q$ and VC dimension $d_V$, the above entropy bound holds with $\rho = 2d_V - 1$ and $\log \Gamma = d_V \log (4e) + \log d_V + K$ for some constant $K$ (see Thm 2.6.4 of \cite{van1996weak}). 
\end{remark}

\noindent
{\bf Proof Sketch.} The generic chaining consists of a clever use of the union bound, subGaussian concentration, and a detailed chain construction.
In order to illustrate how we can obtain an exponential inequality for the max of subGaussian random variables, let us consider the max of iid standard Gaussian random variables, $\{z_i\}_{i=1}^N$.
Notice that by union bound and subGaussian concentration,
$$
\mathbb P \left\{ \max_i z_i > \sqrt{2 \log N + u^2} \right\} %\le N \mathbb P \left\{ z_1 > \sqrt{2 \log N + u^2} \right\}
\le N e^{-\log N - \frac{u^2}{2}} = e^{-\frac{u^2}{2}}.
$$
Now, instead of bounding, $\sqrt{2 \log N + u^2} \le u + (\log N) / u$ as is done in the generic chaining, we will use the bound $\sqrt{2 \log N + u^2} \le \sqrt{2 \log N} + u^2 / (2 \sqrt{2 \log N})$ (we Taylor expand the square root around $2 \log N$ instead of around $u^2$).
Hence, we obtain,
$$
\mathbb P \left\{ 2 \sqrt{2 \log N} \left( \max_i z_i - \sqrt{2 \log N} \right) > u \right\} \le e^{-u}.
$$
This is the main technique that we use for obtaining subexponential bounds from the max of independent subGaussian random variables.

The chain construction refers to a sequence of partitionings of the space $\mathcal I$.  
Given a partition $\mathcal A_k$ of $\mathcal I$, we let $A_k(\iota) \in \mathcal A_k$ be the element that contains the point $\iota \in \mathcal I$ and $\Delta(A_k(\iota))$ be its radius.
In the standard generic chaining, a partition is called admissible if $|\mathcal A_k| \le 2^{2^k}$.
Then the supremum is controlled by uniformly bounding the centers of the single element $A_0 \in \mathcal A_0$, and then the differences between the centers of $A_{k} \in \mathcal A_k$ and their closest centers in at level $k-1$.
The standard result is a bound on the supremum based on a functional of the radii $\Delta(A_k(\iota))$.
For our modifications, the subexponential bound above requires a growing number of independent points to work, so we begin our chain at a deeper level than at $k=0$.
Furthermore, we have to modify the definition of an admissible partition to be $|\mathcal A_k| \le a^{a^k}$ where $a \rightarrow 1$ as $\Gamma \rightarrow \infty$.
Technical details regarding the chain introduce the $\log \log \Gamma$ term.
See the appendix for a complete proof.
Although, we assume a specific form for the covering numbers, it may be possible to generalize this technique to other entropy bounds.

\section{Type 1 error guarantees for learning patterns}
\label{sec:main}

A type 1 error---detecting an anomaly under the null hypothesis---is typically the first error to be controlled in the Neyman-Pearson testing framework.
Our statistic $S_n$ will be compared to a threshold, which is determined through calibration, simulation, or theoretical guarantees.
In this section, we will provide a finite-sample probabilistic bound for the multiscale scan statistic with exponential tail probability.
We will then use this result to obtain a finite sample bound on the PAMSS, $S_n$, which increases logarithmically with the number of functions, $\log |\mathcal F|$.

\begin{lemma}
\label{lem:single_scale}
Suppose that $f \in \mathcal F$ satisfies either \eqref{eq:TVC} or \eqref{eq:AHC}. 
Let $\ell \in \{0,\ldots, \lfloor \log_2 L\rfloor\}^d$, and $\mathcal H_2(\ell) = \times_j [2^{\ell_j}, 2^{\ell_j + 1}]$.  Then when $L \ge L_0$,
\begin{equation}
\mathbb P \left\{ c_1 \cdot \max_{h \in \mathcal H_\ell, t \in \mathcal T_h} v_h\left( (f_h \star {\rm d} X^i)(t) - v_h\right) - a_1 \log \log L > u \right\} \le e^{-u}
\end{equation}
for constants $L_0,a_1,c_1 > 0$ depending on $\gamma, d$ only.
\end{lemma}

\begin{proof}
Let $\mathcal D'$ denote the index set in the above display.
Throughout, let $c_1$ and $a_1$ denote arbitrary constants depending on $d,\gamma$ alone.
By assumption, $\{(f_h \star {\rm d} W)(t) : (t,h) \in \mathcal D'\}$ is a subGaussian random field with canonical distance $\nu_f$.  
An $\epsilon$-net of $\mathcal D$, by definition, will be an $\epsilon$-covering of $\mathcal D' \subset \mathcal D$, so we just need to bound the size of the $\epsilon$-net, $\mathcal D_{\beta,\alpha}' := \mathcal D' \cap \mathcal D_{\beta,\alpha}$.
By construction, 
$$
|\mathcal D_{\beta,\alpha}'| = \sum_{h \in \mathcal H_2(\ell) \cap \mathcal H_\beta} |\mathcal T_{\alpha,h}| \le c_1 \sum_{h \in \mathcal H_2(\ell) \cap \mathcal H_\beta} \prod_j \frac{L}{\alpha h_j} \le c_1 \frac{(L \log_\beta 2)^d}{\alpha^d 2^{\sum_j \ell_j}}.
$$
and we can take $\alpha, \beta$ as specified in the proof of Theorem \ref{thm:epsnet}.
Furthermore, notice that $2^\ell_j$ is within a factor of 2 of any $h_j$ in $\mathcal H_2(\ell)$.
Then, there are constants $\tilde C, \tilde c$ such that 
$
| \mathcal D_{\beta,\alpha}' | \le \tilde C L^d / (h_\bullet \epsilon^{\tilde c}).
$
We can see this because $\alpha^d (\log \beta)^d \le \alpha^d (\beta - 1)^d \le \epsilon^{\tilde c}$ for some constant depending on $d$.
By Theorem \ref{thm:chaining}, we have the subexponential bound for the random variable, 
$$\sup_{(t,h) \in \mathcal D'} y_h ((f_h \star {\rm d}W)(t) - y_h) {\rm ~~where~~}y_h = \sqrt{2 \sum_j \log (L/h_j) + 2\log \tilde C}$$
Notice that $y_h = v_h + O(1 / v_h)$ which gives us the result, along with $\log v_h \le C \log \log L$.
\end{proof}

From here, we simply apply the union bound with the bound in Lemma \ref{lem:single_scale}, for every element in the partitioning $\mathcal H = \cup_\ell \mathcal H_2(\ell)$.
There are on the order of $\log_2 (d \log_2 L) = \log_2 \log L + O(1)$ such elements.
This gives us 
$$
\mathbb P \left\{ c_2 \cdot \frac{s_n(X^i, f)}{\log \log L} - a_2 > u \right\} \le e^{-u},
$$
for some constants $c_2,a_2$ depending on $d,\gamma$.
The $\log \log L$ is reminiscient of the law of iterated logarithm, and indeed this is a LIL result (but in multiple dimensions).
Hence, $z_i(f) := s_n(X^i, f)/\log \log L$ is subexponential with $K := \| z_i(f) \|_{\psi_1}$ only depending on $d, \gamma$ ($\|. \|_{\psi_1}$ is the Orlitz 1-norm).
Thus, by the subexponential Bernstein inequality (Prop.~5.16 in \cite{vershynin2010introduction}),
$$
\mathbb P \left\{ \sum_{i=1}^n z_i(f) \ge t \right\} \le \exp\left( -c_3 \min\left\{ \frac{t^2}{n K^2}, \frac{t}{K} \right\} \right),  
$$
where $c_3$ is an absolute constant.
This gives us our main result (we absorb $c_3$ into $K$ below).

\begin{theorem}
\label{thm:main}
Let $\mathcal F$ be finite and assume that either all functions in $\mathcal F$ satisfy either \eqref{eq:TVC} or \eqref{eq:AHC}.
Let
\begin{equation}
    F_n(\delta) := \left\{ \begin{array}{ll}
    \sqrt{K \log\left( \frac{|\mathcal F|}{\delta} \right)}, &\log |\mathcal F| \le \frac nK + \log \delta\\
    \frac{K}{\sqrt n} \log \left( \frac{|\mathcal{F}|}{\delta} \right), &\log |\mathcal F| > \frac nK + \log \delta
    \end{array}\right.
\end{equation}
then for some constants $K, L_0$ depending on $d,\gamma$, and $L > L_0$,
\begin{equation}
    \mathbb P\left\{ S_n(X,\mathcal F) > F_n(\delta) \cdot \log \log L \right\} \le \delta.
\end{equation}
\end{theorem}

Theorem \ref{thm:main} proves this paper's main hypothesis, that we can learn patterns from a finite dictionary where the type 1 error bound increases logarithmically with $|\mathcal F|$.
In fact if $\log |\mathcal F| = o(n)$ then we have subGaussian concentration of the final test statistic $S_n$.

\section{$\epsilon$-net approximation and type 2 error}
\label{sec:eps}

We have provided a construction of the $\epsilon$-net with parameters $\alpha> 0$ and $\beta > 1$, but we did not specify the selection of either or prove the veracity of our claim that this indeed produces an $\epsilon$-net.
We used the construction of our $\epsilon$-net in the proof of Lemma \ref{lem:single_scale}, so we will be careful in this section to prove the correctness of our construction from first principles.
The following technical lemma is the main driver of these results.

\begin{lemma}
\label{lem:distbd}
There is a constant $C$ depending on $d$ alone such that
\begin{enumerate}
    \item[1.] Suppose that \eqref{eq:TVC} holds for the class $\mathcal F$, then
    $$
    \nu_f((t,h),(t',h'))^2 \le C \gamma_1 \left( \left\| \frac{t - t'}{h} \right\|^2_2 + \left\| \frac{h - h'}{h} \right\|^2_2 + \left( \sqrt{\frac{h_\bullet'}{h_\bullet}} - 1 \right)^2\right).
    $$
    %where $C(h,h') = \max\{\beta - 1,1 - \beta^{-1}\}$ for $\beta = \max_j \max \{ h_j / $
    \item[2.] [\cite{proksch2016multiscale}] Suppose that \eqref{eq:AHC} holds for the class $\mathcal F$, then
    $$
    \nu_f((t,h),(t',h'))^2 \le C \left( \sum_{j=1}^d \left| \frac{t_j-t_j'}{h_j} \right|^{2 \gamma_2} + \sum_{j=1}^d \left| \frac{h_j - h_j'}{\sqrt{h_j h_j'}} \right|^2 + \sum_{j=1}^d \left| \frac{h_j - h_j'}{\sqrt{h_j h_j'}} \right|^{2\gamma_2} \right).
    $$
\end{enumerate}
\end{lemma}

Lemma \ref{lem:distbd}.1 is proven using the fact that the metric $\nu_f$ is the canonical metric for $S_t f_h$ convolved with the Wiener process.
Then by a strategic use of integration by parts we can bound the variance form of $\nu_f$.
It is immediately clear that we can use this result to prove that $\mathcal D_{\beta,\alpha}$ forms an $\epsilon$-net of the space $(\mathcal D, \nu_f)$.

% Given a function $f \in \mathcal F$, define the rasterized function,
% $$
% f^{R}(t) =  
% $$

\begin{theorem}
\label{thm:epsnet}
%The algorithm \ref{alg:epsnet} implicitly constructs an $\epsilon$-net by identifying
Suppose that either one of \eqref{eq:TVC} or \eqref{eq:AHC} holds.
Let $\epsilon > 0$, then there exists a $\beta > 1, \alpha > 0$ depending on $d,\epsilon$ such that for any $(t,h) \in \mathcal D$, there exists $(t',h') \in \mathcal D_{\beta,\alpha}$ (the $\epsilon$-net) with $\nu_f ((t,h),(t',h')) \le \epsilon.$
\end{theorem}

\begin{proof}
Throughout, let $\gamma$ mean either $\gamma_1, \gamma_2$ depending on context.  By the triangle inequality let us bound,
$$
\nu_f((t,h),(t',h')) \le \nu_f((t,h'),(t',h')) + \nu_f((t,h),(t,h')).
$$
In our construction, notice that for any $t$ there is a grid point in $\mathcal T_{\alpha,h'}$ that is within $\alpha h_j'$ from it in dimension $j$.
Hence,
$$
\sum_{j=1}^d \left| \frac{t_j-t_j'}{h_j'} \right|^{2 \gamma_2} \le d \alpha^{2\gamma_2}; \quad \sum_{j=1}^d \left| \frac{t_j-t_j'}{h_j'} \right|^{2} \le d \alpha^2.
$$
Furthermore, by the $\epsilon$-net construction, there exists an $h_j'$ such that $|\log h_j / h_j'| \le \log \beta$ for every $j$.
Thus,
$$
\left| \frac{h_j - h_j'}{\sqrt{h_j h_j'}} \right| \le \beta - 1; \quad \left| \sqrt{\frac{h_\bullet'}{h_\bullet}} - 1 \right| \le \beta^{\frac d2} - 1.
$$
Thus there is a constant, $C$, depending on $\gamma$, such that for $\epsilon$ small enough
$
\alpha = C \epsilon^{1/\gamma_2} 
$
and 
$
\beta = C (1 + \epsilon)^{2/d}
$
is sufficient.
% Each of the bounds in Lemma \ref{lem:distbd}, can be further bounded by 
% $$
% \nu_f((t,h),(t',h'))^2 \le C_\gamma \left(\left\| \frac{t-t'}{h}\right\|^2_2 + \left\| \frac{h - h'}{\sqrt{h h'}} \right\|_2^2 +  \right)
% $$
\end{proof}

Suppose that we are under the alternative hypothesis, $H_1$, so that there is some embedded signal $f$ in each image and that the noise is a standard Wiener process.
Consider evaluating the scan at the true location and scale, $(t^i,h^i)$ for a given field $X^i$, then
$(f_{h^i} \star {\rm d} X^i)(t^i)$ is normally distributed with mean $\mu$ and variance $1$.
In the event that we scan over an $\epsilon$-net, then by arguments in Section \ref{sec:method}, there is an element in the approximate scan with mean $\mu(1 - \epsilon^2 /2)$.
Hence, we have the following type 1 error bound, by summing the resulting $n$ normal random variables.

\begin{proposition}
Suppose that $\{X^i\}_{i=1}^n$ are drawn from $H_1$ (with possibly different location and scale parameters) where the noise random field, $W^i$, is a standard Wiener process in $d$ dimensions.
Then define $V_n = \sum_i v_{h^i}^2$ and $M_n = \sum_i v_{h^i}$,
$$
\mathbb P \left\{ E_{n}(X;\mathcal F) - \frac{1}{\sqrt n} \left( \mu \cdot \left(1 - \frac{\epsilon^2}{2}\right) M_n - V_n \right) < u \sqrt{\frac{V_n}{n}} \right\} \le \Phi(u)
$$
where $\Phi$ is the standard Normal CDF.  
The above display is also true if we let $\epsilon = 0$ and substitute $E_{n} \gets S_n$.
\end{proposition}

In order to have diminishing type 1 error probability, we set a threshold for $S_n$ at $F_n(\delta) \cdot \log \log L$ for $\delta \rightarrow 0$.
Assume that $\epsilon \rightarrow 0$ (however slowly) then to have the type 2 error probability to decrease as well, we require that
\[
\frac{\mu M_n - V_n}{\sqrt n} - F_n(\delta) \log \log L = \omega\left( \sqrt{\frac{V_n}{n}} \right).
\]
For comparison purposes, let's derive other conditions for the PAMSS to be asymptotically powerful (diminishing type 1 and 2 error).
Suppose that $n = 1$ and $|\mathcal F| = 1$ (the standard multiscale scan setting), then this would require
\[
\mu - v_{h^1} - \frac{K}{v_{h^1}} \log \frac 1 \delta \cdot \log \log L \rightarrow \infty,
\]
which is consistent up to constant $K$ and lower order additive terms with previously known rates, \cite{proksch2016multiscale}.
Critically, the dependence on $\delta$ is logarithmic, and this provides the first such exponential finite sample bound (\cite{dumbgen2001multiscale} provided only an almost sure bound; \cite{proksch2016multiscale} gave a polynomial bound).

We can see that if $|\mathcal F| = 1$ and $n \rightarrow \infty$ then $F_n(\delta) = O(1)$; if in addition, $h_j^i \le L^c$ for some $0 \le c < 1$, then $V_n / n = \Omega(\log L)$ and $\log \log L$ is lower order.
Then if these conditions hold, then a condition for the PAMSS to be asymptotically powerful is
$
\mu - V_n/M_n\rightarrow \infty.
$
We find that we can make a similar conclusion even if $|\mathcal F|$ grows in $n$.

\begin{corollary}
Suppose that $\log |\mathcal F| = o(n)$, and recall that under the alternative hypothesis, $H_1$, $X^i$ has an embedded pattern $f$ at scale $h^i$ and $v_{h^i}^2 = \sum_j \log(L/h_j^i)$, and the noise is a standard Wiener process.
Suppose also that $h^i_j \le L^{c}$ for some $0 \le c < 1$ for all $i,j$, then the PAMSS is asymptotically powerful (has diminishing probability of type 1 and type 2 error) if
%\begin{equation}
%\mu - \sqrt 2 \cdot \frac{\sum_{i=1}^n \sum_{j=1}^d \log \frac{L}{h_j^i}}{\sum_{i=1}^n \sqrt{ \sum_{j=1}^d \log \frac{L}{h_j^i}}} \rightarrow \infty.
%\end{equation}
\begin{equation}
\mu - \sqrt 2 \cdot \frac{\sum_{i=1}^n v_{h^i}^2}{\sum_{i=1}^n v_{h^i}} \rightarrow \infty.
\end{equation}
We take this result to mean that as long as the function class, $|\mathcal F|$, does not grow exponentially in $n$, we achieve asymptotic power under the same conditions as if $|\mathcal F| = 1$.
\end{corollary}

\begin{proof}
We can see this because, under the assumptions, $V_n \ge C n \log L$ for some $C > 0$ and $\sqrt{V_n} / M_n \rightarrow 0$.
Furthermore, $F_n(\delta) = o(\sqrt{n})$ and so $M_n / \sqrt n = \omega(F_n(\delta) \log \log L)$ making the term involving $|\mathcal F|$ lower order.
Evaluating $M_n,V_n$ gives us the result.
\end{proof}
% Then the 
% $$
% \frac{\mu}{2} - \frac{V_n}{W_n} = \omega\left( \frac{\sqrt{V_n}}{W_n} \cdot F_n(\delta) \cdot \log \log L \right)
% $$
% $$
% \frac{\mu}{2} - \sqrt{2 \log n} = \omega \left(\sqrt{\frac{\log |\mathcal F|}{n}} \cdot \log \log L \right).
% $$

Let us conclude with a remark about the restrictiveness of the assumption that we have a finite function class $\mathcal F$.
It is known that functions of bounded variation have Haar wavelet coefficients that are bounded in a weak $\ell_1$ norm, \cite{cohen2003harmonic}.
It is reasonable to discretize the allowed coefficient values and then restrict our function class to functions with $k$-sparse wavelet coefficients of $m$ then the log-size of the class scales like $k \log m$ which is very manageable.
One advantage with this approach is that the sparse Haar wavelets will naturally satisfy condition \eqref{eq:TVC}.
It is outside of the scope of this work to extend the result to infinite function classes, but this would present a very interesting and important extension.

\section{Conclusions}

We have addressed learning and detecting patterns from a function class, $\mathcal F$, using multiscale scan statistics.
We have introduced the multiscale scan statistic and proved a subexponential concentration bound for it, which relied on a novel chaining result for standardized suprema of subGaussian random fields (a result that may be of independent interest).
We introduced the pattern adapted multiscale scan statistic, that can learn patterns in a database of tensors (when the locations and scales vary).
This result allowed us to prove Bernstein-type concentration for the PAMSS, meaning that we can learn finite function classes that grow exponentially with the sample size, $n$.
With evidence that representation learning and detection are not incompatible, we anticipate that efficient methods for learning functions in this setting will emerge, by using modern tools from deep learning and multiscale methods.

% $$
% \frac{\mu}{2} \frac{W_n}{\sqrt{V_n}} - \sqrt{V_n} \rightarrow \infty
% $$

% Consider evaluating the scan at the true location and scale, $(t^i,h^i)$ for a given field $X^i$, then
% $(f_{h^i} \star {\rm d} X^i)(t^i)$ is normally distributed with mean $\mu$ and variance $1$.

% $$
% \frac{\mu}{2} - \frac{V_n}{W_n} = \omega\left( \frac{\sqrt{V_n}}{W_n} \cdot F_n(\delta) \cdot \log \log L \right)
% $$
% $$
% \frac{\mu}{2} - \sqrt{2 \log n} = \omega \left(\sqrt{\frac{\log |\mathcal F|}{n}} \cdot \log \log L \right).
% $$

% Acknowledgments---Will not appear in anonymized version

\noindent
{\bf Acknowledgements:} JS is supported in part by NSF DMS-1712996.

\bibliography{detectnets}

\appendix

\section{Proof of Theorem \ref{thm:chaining}}

\begin{proof}
We will follow the construction of generic chains as in \cite{talagrand2006generic}, but will be significantly more careful about the details of the construction.
This proof is similar in spirit to \cite{dumbgen2001multiscale}, but we will get probabilistic bounds and prefer this proof because it uses only first principles.
Throughout this proof we will call variables that dependent only on $\rho$, constants, and some like $C$ may change from line to line.
Variables from the main body of the paper, other than those defined in Theorem \ref{thm:chaining}, may appear in this proof and mean something different (we suppose that they are in a different scope).

Let's begin by defining $G = \Gamma^{1/\rho}$ and
\begin{equation*}
    a = (1 - \log^{-1} G)^{-1} > 1.
\end{equation*}
Let an admissible partition, $\mathcal A_k$, be any partition of $\mathcal I$ of size at most $a^{a^k}$.
Let $A_k(\iota)$ be the element of the partition containing $\iota$ and let the center of this element be 
\begin{equation*}
\tau_n(\iota) := \inf_{\tau \in A_n(\iota)} \sup_{\iota' \in A_n(t)} d_Z(\tau, \iota'),  
\end{equation*}
and $\Delta(A_k(t)) = \sup_{\iota \in A_k(t)} d_Z(\tau_k(\iota),\iota)$ be its radius.
Let $k_0 = 1,\ldots,\infty$ then $Z(\iota) \le Z(\tau_{k_0}(\iota)) + \sum_{k=k_0}^\infty |Z(\tau_{k+1}(\iota)) - Z(\tau_{k}(\iota))|$.
By the union bound and \eqref{eq:subGaussian2},
\begin{align*}
    \mathbb P \left\{\exists \iota \in \mathcal I : Z(\tau_{k_0}(\iota)) > \sqrt{2 \log a \cdot a^{k_0} + u^2} \right\} &\le a^{a^{k_0}} \exp \left(-\frac 12 \left(2 \log a \cdot a^{k_0} +  u^2 \right)\right) \\
    &= e^{-\frac{u^2}{2}}.
\end{align*}
Let 
\begin{equation}
\label{eq:epsilon}
    \epsilon_{k} = G \cdot a^{- \frac{a^k}{\rho}},
\end{equation} 
then there exists an admissible partition where the radius of the balls are $\epsilon_k$ satisfy $\mathcal N(\mathcal I, d_Z, \epsilon_k) \le a^{a^k}$.
Hence,
\begin{align*}
    &\mathbb P \left\{\exists \iota \in \mathcal I : |Z(\tau_{k + 1}(\iota)) - Z(\tau_k(\iota))| > \epsilon_k \sqrt{2 \log a \cdot a^{k+1} + d_k u^2} \right\} \\
    &\quad \le 2 a^{a^{k+1}}\exp \left(-\frac 12 \left(2 \log a \cdot a^{k+1} +  d_k u^2 \right)\right) \\
    &\quad = 2 e^{-\frac{d_k u^2}{2}},
\end{align*}
where
\begin{equation}
\label{eq:ddef}
    d_k = e^{\frac 12 \left( 1 + \log a \right) \cdot (k-k_0)}.
\end{equation}
Define the quantities, 
\begin{align*}
    A_{k_0} &= \sqrt{2 \log a \cdot a^{k_0} + u^2} + \sum_{k=k_0}^{\infty} \epsilon_k \sqrt{2 \log a \cdot a^{k+1} + d_k u^2},\\
    B_{k_0} &= \exp \left( - \frac{u^2}{2} \right) + \sum_{k=k_0}^\infty \exp \left(- d_k \frac{u^2}{2} \right),
\end{align*}
so that 
\begin{equation}
\label{eq:Aprobbd}
    \mathbb P \{ \sup_{\iota \in \mathcal I} Z(\iota) > A_{k_0} \} \le 2 B_{k_0}.
\end{equation}
Because $\sqrt{b_1 + b_2} \le \sqrt{b_1} + b_2 / (2 \sqrt{b_1})$ for $b_1,b_2 > 0$ we have 
\begin{equation}
\label{eq:Abound}
    A_{k_0} \le a_0 a^{k_0/2} + \frac{a^{-k_0/2}}{2 a_0} u^2 + \sum_{k=k_0}^\infty \left(a_0 \epsilon_k a^{\frac{k+1}{2}} + \frac{a^{- \frac{k+1}{2}}}{2 a_0} \epsilon_k d_k u^2 \right)
\end{equation}
for $a_0 = \sqrt{2 \log a}$.
Let $k_0$ satisfy
\begin{equation}
\label{eq:k0ineq}
    a^{k_0} \le \rho (\log_a G + k_0) \le a^{k_0 + 1},
\end{equation}
such that
$$
k_0 \log a \le C' \log \log G
$$
for some constant $C'$ possibly depending on $\rho$,
which is guaranteed by Lemma \ref{lem:k0exists} for $\Gamma$ large enough.
Then we have that 
\begin{equation}
\label{eq:Aterm1}
    a_0 a^{\frac{k_0}{2}} \le \sqrt{2 \log \Gamma + 2 \rho k_0 \log a} \le \sqrt{2 \log \Gamma} + \frac{\rho k_0 \log a}{\sqrt{2 \log \Gamma}} \le \sqrt{2 \log \Gamma} + \frac{C' \log \log \Gamma}{\sqrt{2 \log \Gamma}}
    %\sqrt{2 \log \Gamma + 2 \rho \frac{k_0}{\log \Gamma}} \le \sqrt{2 \log \Gamma} + \frac{k_0 \rho}{\sqrt 2 (\log \Gamma)^{3/2} },
\end{equation}
Furthermore, by \eqref{eq:epsilon} and \eqref{eq:k0ineq}
\begin{align*}
    \sum_{k=k_0}^\infty a^{\frac{k+1}{2}} \epsilon_k &\le G \sum_{k=k_0}^{\infty} a^{-\frac{a^k}{\rho} + \frac{k+1}{2}} = G a^{\frac{k_0 + 1}{2}} \sum_{j=0}^{\infty} a^{-a^{k_0} \frac{a^j}{\rho} + \frac{j}{2}}\\
    &\le G a^{\frac{k_0 + 1}{2}} \sum_{j=0}^\infty a^{-(\log_a G + k_0) a^{j-1} + \frac j2} = a^{\frac{1 - k_0}{2}} \sum_{j=0}^\infty a^{(\log_a G + k_0) (1 - a^{j-1}) + \frac j2}\\
    &= a^{\frac{1 - k_0}{2}} \left( \left(G a^{k_0}\right)^{1 - a^{-1}} + \sqrt a \sum_{j=0}^\infty a^{(\log_a G + k_0) (1 - a^{j}) + \frac j2} \right).
\end{align*}
The first term is
\begin{equation*}
    a^{\frac{1 - k_0}{2}} \left( G a^{k_0} \right)^{1 - a^{-1}} = e \cdot a^{\frac{1 - k_0}{2}} a^{\frac{k_0}{\log G}}.
\end{equation*}
Furthermore, for any $C < 1$ there exists a $G_0 > 0$ such that $\log a \ge C / \log G$ for all $G \ge G_0$, thus
\begin{equation*}
    a^{\frac{k_0}{\log G}} \le \left( \rho (\log_a G + k_0) \right)^{\frac{1}{\log G}} \le \left( \rho (C \log^2 G + k_0) \right)^{\frac{1}{\log G}} \rightarrow 1,
\end{equation*}
as $G \rightarrow \infty$.
Considering the second term,
\begin{equation*}
    \sum_{j=0}^\infty a^{(\log_a G + k_0)(1 - a^j) + \frac j2} \le 1 + \sum_{j=1}^\infty a^{(\log_a G)(1 - a^j) + \frac j2}.
\end{equation*}
Using the first order approximation,
\begin{equation*}
    a^j - 1 = \left( 1 - \frac{1}{\log G} \right)^{-j} - 1 \ge \frac{j}{\log G}.
\end{equation*}
So we may bound the series,
\begin{equation*}
\sum_{j=0}^\infty a^{(\log_a G) (1 - a^j) + \frac j2} \le \sum_{j=0}^\infty \exp\left( \left(\frac{\log a}{2} - 1 \right) j \right) \le \left(1 - \exp\left( \frac{\log a}{2} - 1 \right) \right)^{-1}\rightarrow \frac{1}{1 - e^{-1}},
\end{equation*}
as $G \rightarrow \infty$.
Hence, there are constants $C, G_0 > 0$ depending on $\rho$ such that for $G > G_0$, 
\begin{equation}
\label{eq:Aterm2}
    \sum_{k=k_0}^\infty a_0 \epsilon_k a^{\frac{k+1}{2}} \le C a_0 a^{-\frac{k_0}{2}} \le C \sqrt{\frac{2 a \log a}{\rho (\log_a G + k_0)}} \le \frac{C \sqrt{2 a} \log a}{\sqrt{2 \rho \log G}}.
\end{equation}
Also, 
\begin{equation}
\label{eq:Aterm3}
    \frac{a^{-\frac{k_0}{2}}}{a_0} \le \frac{C}{\sqrt{2 \rho \log G}}.
\end{equation}
Consider the fourth term in $A_{k_0}$,
\begin{align*}
    &\sum_{k=k_0}^\infty a^{-\frac{k+1}{2}} \epsilon_k d_k = G \sum_{k=k_0}^\infty a^{-\frac{k+1}{2}} a^{-\frac{a^k}{\rho}} e^{\frac 12 (1 + \log a) \cdot (k-k_0)} \\ 
    & \le G a^{-\frac 12 (k_0 + 1)} \sum_{j=0}^\infty a^{-\frac{j}{2}} a^{-(\log_a G + k_0) a^{j-1}} e^{\frac 12 (1 + \log a) \cdot j} \\
    & \le a^{- \frac 32 k_0 - \frac 12} \sum_{j=0}^\infty a^{-\frac{j}{2}} a^{(\log_a G + k_0) \cdot (1 - a^{j-1})} e^{\frac 12 (1 + \log a) \cdot j}
\end{align*}
Isolating the summation,
\begin{align*}
    \sum_{j=1}^\infty a^{(\log_a G + k_0) \cdot (1 - a^{j-1})} e^{\frac j2} &\le \sqrt e \sum_{j=0}^\infty \exp\left( \left( \frac 12 - \frac{\log G + k_0 \log a}{\log G} \right) \cdot j \right)\\
    & \le \frac{\sqrt e}{1 - 1/\sqrt e} \le 4.2. \\
\end{align*}
Furthermore,
\begin{equation*}
    a^{(\log_a G + k_0) \cdot (1 - a^{-1})} \le e^{-\left( 1 + \frac{k_0}{\log_a G} \right) } \le \frac 1e < 0.4.
\end{equation*}
Hence, there exists a $G_0 > 0$ such that 
\begin{equation}
\label{eq:Aterm4}
    \frac{u^2}{a_0} \sum_{k=k_0}^\infty a^{-\frac{k+1}{2}} \epsilon_k d_k \le u^2 \frac{4.6}{\sqrt{2 \log a}} \cdot a^{-\frac 32 k_0 - \frac 12} \le \frac{u^2}{\sqrt{\rho \log G}},
\end{equation}
for $G \ge G_0$.
Hence, by \eqref{eq:Abound} combined with \eqref{eq:Aterm1}, \eqref{eq:Aterm2}, \eqref{eq:Aterm3}, and \eqref{eq:Aterm4} yields
\begin{equation*}
    A_{k_0} \le \sqrt{2 \log \Gamma} + \frac{a_0 \log \log \Gamma + u^2/2}{\sqrt{c_0 \log \Gamma}},
\end{equation*}
for constants $C_0, C_1, G_0$ with $G \ge G_0$.
Consider $B_{k_0}$,
\begin{align}
    &\exp \left( \frac{u^2}{2} \right) \cdot \sum_{k=k_0}^\infty \exp \left( - d_k \frac{u^2}{2} \right) = \sum_{j=0}^\infty \exp \left( \left(1 - \exp \left( \frac 12 (1 + \log a)\cdot j \right) \right) \cdot \frac{u^2}{2} \right) \\
    &\quad \le \sqrt e \sum_{j=0}^\infty \exp \left( - \frac 12 \exp\left( \frac 12 j \right) \right) = \sqrt e \sum_{j=0}^\infty b^{-b^j} \le 2.4, \\
\end{align}
where $b = \sqrt e$.
Hence,
\begin{equation*}
    B_{k_0} \le C \cdot e^{-\frac 12 u^2}.
\end{equation*}
Recalling \eqref{eq:Aprobbd} gives us our result.
\end{proof}

\begin{lemma}
\label{lem:k0exists}
Let $\rho > 0$, and define the function 
\begin{equation}
    a(G) = \left( 1 - \frac{1}{\log G} \right)^{-1}.
\end{equation}
There exists a constant $G_0$ depending on $\rho$ alone such that if $G > G_0$ then there exists an integer $k_0$ satisfying
\begin{equation}
    a^{k_0} \le \rho (\log_a G + k_0) \le a^{k_0 + 1},
\end{equation}
and there is a constant $C$ depending on $\rho$ such that $k_0 \log a \le C \log \log G$.
\end{lemma}

\begin{proof}
Set $G > 0$ and $a = a(G)$.
Let $\kappa_1$ be the root of the following function,
\begin{equation*}
    f(\kappa) := a^{\kappa} - \rho ( \log_a G + \kappa ).
\end{equation*}
Notice that $f$ is strictly convex and so let $\kappa_0$ be its unique minimizer.
\begin{equation*}
    \kappa_1 = \log_a \left( \rho (\log_a G + \kappa_1 ) \right) \ge \log_a \left( \rho \log_a G \right),
\end{equation*}
Notice that $\log a \le 1/\log G$, hence,
\begin{equation*}
    \log_a G = \frac{\log G}{\log a} \ge \log^2 G,
\end{equation*}
and
\begin{equation*}
    \kappa_1 \ge \frac{\log \rho}{\log a} + 2 \frac{\log \log G}{\log a} \ge \log G \cdot (\log \rho + 2 \log \log G) =: x_1.
\end{equation*}
Moreover, $\kappa_0$ satisfies $\kappa_0 a^{\kappa_0} = \rho$,
\begin{equation*}
    (\kappa_1 - 1) a^{\kappa_1 - 1} =  \frac{\rho}{a} (\kappa_1 - 1)(\log_a G + \kappa_1) \ge  \frac{\rho}{a} (x_1 - 1) (\log^2 G + x_1).
\end{equation*}
The limit of the right hand side approaches $+\infty$ as $G$ increases.
Hence, there exists an $G_0 > 0$ depending on $\rho$ only such that for any $G \ge G_0$,
\begin{equation*}
    (\kappa_1 - 1) a^{\kappa_1 - 1} \ge \rho = \kappa_0 a^{\kappa_0}.
\end{equation*}
Because $\kappa a^\kappa$ is a monotone increasing function of $\kappa > 0$, this implies that $\kappa_1 - 1\ge \kappa_0$.

Let $k_0 = \lfloor \kappa_1 \rfloor$.
We know that $f$ is increasing and convex over $[\kappa_0,\infty)$.
Let $G \ge G_0$.
Because $\kappa_0 \le k_0 \le \kappa_1 \le k_0 + 1$ we have that $f(k_0) \le 0 \le f(k_0 + 1)$.
Hence,
\begin{equation*}
    a^{k_0} - \rho (\log_a G + k_0) \le 0 \le a^{k_0 + 1} - \rho (\log_a G + k_0 + 1) \le a^{k_0 + 1} - \rho (\log_a G + k_0),
\end{equation*}
which proves the first result.
Finally, 
$$
\log(a^{2 k_0}) \le \log\left( a^{k_0} - \rho k_0 \right) \le \log(\rho \log_a G)
$$
for some $G$ ($k_0$) large enough.
Notice that $\log_aG \le C (\log G)^2$ by the limit of $\log a \rightarrow 1 / \log G$, demonstrating the second result.
\end{proof}

\section{Remaining proofs}

{\bf A note on the derivation of Theorem \ref{thm:main}}.
Notice that $t^2 / nK < t/K$ when $t < n$, and that let
$$
t = \sqrt{\frac{nK}{c_3} \log \frac{|\mathcal F|}{\delta}}.
$$
Then
\begin{align*}
    |\mathcal F| \exp \left( - c_3 \frac{t^2}{nK^2} \right) \le \delta
\end{align*}
and $t < n$ when 
$$
\log |\mathcal F| \le \frac{n}{c_3 K} + \log \delta.
$$
In the other case, we define $t$ so that the probabilistic bound holds, and absorb $c_3$ into K.

\begin{proof}{\bf of Lemma \ref{lem:distbd}.}
(2) is proven in \cite{proksch2016multiscale} (pg. 32) so we will focus on (1).
We will partially follow the arguments in \cite{dumbgen2001multiscale} (pg. 145).
Let $W$ denote a standard Wiener process in $d$ dimensions.  Notice that we can define $\nu_f$ as,
$$
 \nu_f((t,h),(t',h'))^2 := \mathbb V \left( \int_{\Omega_L} f_h(t-z) {\rm d}W(z) - \int_{\Omega_L} f_{h'}(t'-z) {\rm d}W(z) \right).
$$

%Suppose that $\beta^{-1} \le |h_j / h_j'| \le \beta$ and $\beta < 1$,
Recall that
$$
(f_h \star {\rm d}W)(t) = \frac{1}{\sqrt{h_\bullet}}\int_{\Omega} f(u) {\rm d} W(t - h u).  
$$
By integration by parts,
$$
\frac{1}{\sqrt{h_\bullet}}\int_{\Omega} f(u) {\rm d} W(t - h u) = \int_\Omega \frac{W(t - h u)}{\sqrt{h_\bullet}} \cdot \nabla f(u) {\rm d} u.   
$$
Therefore,
\begin{align*}
&(f_h \star {\rm d}W)(t) - (f_{h'} \star {\rm d}W)(t') = \int_\Omega \left(\frac{W(t - h u)}{\sqrt{h_\bullet}} - \frac{W(t' - h' u)}{\sqrt{h_\bullet'}}\right) \cdot \nabla f(u) {\rm d} u\\
& \le \int_\Omega \left\|\frac{W(t - h u)}{\sqrt{h_\bullet}} - \frac{W(t' - h' u)}{\sqrt{h_\bullet'}}\right\|_2\cdot \| \nabla f(u) \|_2 {\rm d}u\\
& \le \sup_{u \in \Omega} \left\|\frac{W(t - h u)}{\sqrt{h_\bullet}} - \frac{W(t' - h' u)}{\sqrt{h_\bullet'}}\right\|_2 \cdot \int_\Omega \| \nabla f(u) \|_2 {\rm d}u.
\end{align*}
Hence,
$$
\nu_f((t,h),(t',h'))^2 \le \| f\|_{\rm TV}^2 \cdot \mathbb E \sup_{u \in \Omega} \left\|\frac{W(t - h u)}{\sqrt{h_\bullet}} - \frac{W(t' - h' u)}{\sqrt{h_\bullet'}}\right\|_2^2.
$$
Decompose the supremum term,
\begin{align*}
    & \mathbb E\sup_{u \in \Omega} \left\|\frac{W(t - h u)}{\sqrt{h_\bullet}} - \frac{W(t' - h' u)}{\sqrt{h_\bullet'}}\right\|_2 \\
    &\le \mathbb E \sup_{u \in \Omega} \left( \left\| \frac{W(t-hu) - W(t' - h' u)}{\sqrt{h_\bullet}} \right\|_2 + \|W(t'-h'u)\|_2 \left| \frac{1}{\sqrt{h_\bullet}} - \frac{1}{\sqrt{h_\bullet'}} \right| \right)
\end{align*}
By Brownian scaling,
$
W(t-hu) - W(t' - h' u))/\sqrt{h_\bullet} \overset{d}{=} W\left( \frac th - u \right) - W\left(\frac{t'}{h} - \frac{h'}{h} u\right).
$
Hence, the $j$th coordinate in the above LHS is a 1-dimensional Brownian motion equal in distribution to, 
$
B\left( \frac{t_j - t'_j}{h_j} - \left( 1 - \frac{h'_j}{h_j}\right) u_j \right).
%\le  |u_j|.
$
By the reflection principle, we have that for some constant $c'$,
$$
\mathbb E \sup_{u \in \Omega} B\left( \frac{t_j - t'_j}{h_j} - \left( 1 - \frac{h'_j}{h_j}\right) u_j \right)^2 \le c'\left( \left| \frac{t_j - t_j'}{h_j} \right|^2 + \left| \frac{h_j - h_j'}{h_j} \right|^2 \right).
$$
Hence,
$$
\mathbb E \sup_{u \in \Omega} \left\| \frac{W(t-hu) - W(t' - h' u)}{\sqrt{h_\bullet}} \right\|_2^2 \le c' \left(\sum_j \left| \frac{h_j - h_j'}{h_j} \right|^2 +  \left| \frac{t_j - t_j'}{h_j} \right|^2 \right).
$$
Without loss of generality, we can select $t = 0$ by translation invariance of $\nu_f$.
Furthermore by a similar argument, $\mathbb E \sup_u \|W(h' u)\|_2^2 \le c' h_\bullet'$, for another universal constant $c'$.
Hence, 
$$
\mathbb E \sup_{u \in \Omega} \|W(t'-h'u)\|_2 \left| \frac{1}{\sqrt{h_\bullet}} - \frac{1}{\sqrt{h_\bullet'}} \right| \le c' \left| \sqrt{\frac{h_\bullet'}{h_\bullet}} - 1 \right|.
$$
Combining these bounds completes the proof.
\end{proof}

\end{document}